\documentclass[12pt]{amsart}
\usepackage[top=1in, bottom=1in, left=1in, right=1in]{geometry}
\usepackage{amsfonts}
\usepackage{amsmath}
\usepackage{amssymb}
\usepackage{bbm}
\numberwithin{equation}{section}
\usepackage{epsfig}
\usepackage{graphicx}
\usepackage{tabularx}
\usepackage{times}
\usepackage[usenames,dvipsnames]{color}
\usepackage{comment}
\usepackage{mathtools}
\usepackage{bm}
\usepackage{esvect}

\usepackage{hyperref}
\usepackage{breakurl}
\hypersetup{
colorlinks = true,
linkcolor={black},
urlcolor={blue},
citecolor={blue},    
urlcolor = {blue},
citebordercolor = {0.33 .58 0.33},
linkbordercolor = {0.99 .28 0.23}}

\newcommand{\kommentar}[1]{}

\newtheorem{thm}{Theorem}[section]

\newtheorem{lem}[thm]{Lemma}
\newtheorem*{hypothesis*}{Hypothesis}

\theoremstyle{remark}
\newtheorem{rem}{Remark}[section]

\title{A Conditional Refinement of Page's Theorem on zeros of Dirichlet $L$-functions}
\author{Debmalya Basak and Kyle Pratt}
\address{
Debmalya Basak: Department of Mathematics,
University of Illinois Urbana-Champaign,
Urbana, IL, 61801, USA}
\email{dbasak2@illinois.edu}

\address{
Kyle Pratt: Brigham Young University, Department of Mathematics, Provo, UT 84602, USA}
\email{kyle.pratt@mathematics.byu.edu}

\begin{document}

\begin{abstract}
Landau--Siegel zeros are hypothetical zeros of Dirichlet $L$-functions that are close to the point $s=1$. A classic theorem of Page shows at most one such zero can exist among all Dirichlet $L$-functions with conductor $\leq Q$. We show that one can significantly refine Page's theorem under the assumption that all non-real zeros of Dirichlet $L$-functions lie outside a shrinking neighborhood of $s=1$.

\end{abstract}

\maketitle
\setcounter{tocdepth}{1}
\tableofcontents
\section{Introduction}\label{sec : Introduction}

Determining the location of the zeros of an $L$-function is an important, and largely unsolved, problem in number theory. Any ``reasonable'' $L$-function should satisfy an appropriate Riemann Hypothesis, which would precisely locate the zeros, but unconditionally we can only prove weak results on regions where zeros cannot exist. These zero-free regions go back to the work of de la Vall\'ee Poussin and Hadamard on the Riemann zeta function, and are now known for broad classes of $L$-functions (see, e.g., \cite[Theorem 5.10]{IK2004}). 

We focus in this work on Dirichlet $L$-functions. Let $\chi$ be a primitive Dirichlet character of conductor $q\geq 3$, and let $L(s,\chi)$ denote the associated $L$-function. In this case, the zero-free region for $L(s,\chi)$ takes the following form: there is an absolute constant $c_0>0$ such that if $s = \sigma+it$, then $L(s,\chi)$ has at most a single zero in the region
\begin{align}\label{Eq: Classical ZFR}
    \sigma \geq 1 - \frac{c_0}{\log(q(|t|+2))}.
\end{align}
(See \cite[p. 93]{Davenport}.) If the zero exists in this region, then it is necessarily real and the character $\chi$ is quadratic. These exceptional zeros are known as Landau--Siegel zeros, and the associated Dirichlet characters are known as exceptional characters. Of course, the Generalized Riemann Hypothesis implies that no Landau--Siegel zeros exist, but showing unconditionally that no such zeros exist seems to be extremely challenging. Landau--Siegel zeros are a persistent annoyance to number theorists, since the zeros distort the distribution of primes in arithmetic progressions (see, e.g. \cite[Theorem 5.27]{IK2004}).

It follows from Siegel’s lower bound \cite{Siegel1935} on $L(1, \chi)$ that for any $\varepsilon > 0$, there exists a constant $c(\varepsilon)>0$ such that if the exceptional real zero $\beta$ exists, then
\begin{align}\label{Eq: ZFR Siegel}
\beta \leq 1-c(\varepsilon)q_{\chi}^{-\varepsilon}.
\end{align}
Unfortunately, for $\varepsilon< \frac{1}{2}$, no
known proof provides an effective determination of $c(\varepsilon)$ in terms of $\varepsilon$. Define
\[
\mathcal{S} : = \{ \chi \bmod q_{\chi} : \chi \textrm{ primitive and real} \}.
\]
Tatuzawa \cite{Tatuzawa} refined Siegel's lower bound on $L(1,\chi)$. Using \cite[Lemma 1]{Hoffstein} to translate a
lower bound on $L(1, \chi)$ to an upper bound on $\beta$, Tatuzawa’s work shows that for all $\varepsilon>0$, there exists an effectively computable constant $c_3(\varepsilon)>0$ such that
\begin{align} \label{Eq: ZFR Tatuzawa}
\#\left\{\chi \in \mathcal{S}: q_\chi \geq c_3(\varepsilon) \text { and } L(s, \chi) \text { has a real zero in }\left[1-q_\chi^{-\varepsilon}, 1\right)\right\} \leq 1 .
\end{align} 
\par
It is plausible that the zero-free region in \eqref{Eq: ZFR Tatuzawa} can be improved under some additional assumptions on the non-real zeros of $L(s,\chi)$. Consider the hypothesis that if $\chi \in \mathcal{S}$, then all zeros of $L(s, \chi)$ lie on $\operatorname{Re}(s)=1 / 2$ or $\operatorname{Im}(s)=0$. In other words, suppose the Generalized Riemann Hypothesis is assumed to hold only for the non-real zeros, so Landau--Siegel zeros are permitted to exist. Under this hypothesis, Sarnak and Zaharescu \cite[Proof of Theorem 1]{Sarnak-Zaharescu} established the following improvement of Tatuzawa's result: for any $\varepsilon>0$, there exists an effectively computable constant $q_0 = q_0(\varepsilon)>0$ depending only on $\varepsilon$ such that
\begin{align}\label{Eq: Sarnak-Zaharescu ZFR}
\#\left\{\chi \in \mathcal{S}: q_\chi \geq q_0 \text { and } L(s, \chi) \text { has a real zero in }\left[1-\left(\log q_\chi\right)^{-\varepsilon}, 1\right)\right\} \leq 1.
\end{align}
This “exponentiates” the quality of the zero-free regions in \eqref{Eq: ZFR Siegel} and \eqref{Eq: ZFR Tatuzawa}, albeit under a strong assumption regarding the non-real zeros of Dirichlet $L$-functions. The core idea in their approach is the use of the explicit formula combined with positivity arguments. Since extending positivity arguments to complex domains is difficult, the assumption that all non-real zeros of $L(s,\chi)$ lie on the critical line $\operatorname{Re}(s) = \frac{1}{2}$ becomes essential. For a more detailed discussion of this method, see \cite[Section 5]{Iwaniec}.

In \cite{basak2024remarks}, Thorner, Zaharescu, and the first author  used Tur\'an's power sum method to show that \eqref{Eq: Sarnak-Zaharescu ZFR} holds under a much
weaker hypothesis. Fix $0 < \delta < 1/10$.
\begin{hypothesis*}[$H_{\delta}$]
\textup{If $\chi\in \mathcal{S}$, then all the zeros of $L(s,\chi)$ in the disk $|z-1| < \delta$ are real.}
\end{hypothesis*}
Under Hypothesis $H_{\delta}$, it was shown in \cite{basak2024remarks} that for all $\varepsilon>0$, there exists an effectively computable constant $q_0=q_0(\delta, \varepsilon)>0$ such that
\begin{align}\label{Eq: ZFR BTZ}
\#\{\textup{$\chi\in \mathcal{S}$: $q_{\chi}\geq q_0$ and $L(s,\chi)$ has a real zero in $[1-(\log q_{\chi})^{-\varepsilon},1)$}\}\leq 1.
\end{align}

In light of the ongoing nuisance that is Landau--Siegel zeros, and the difficulty in showing these zeros do not exist, there is great interest in circumventing or limiting their influence. A classical result in this direction is due to Page \cite[Lemma 9]{page1935number} (see also \cite[p. 95]{Davenport}), which analyzes this phenomenon for conductors up to a specified range. For $Q \geq 2$, define
\[
\mathcal{S}(Q)=\{\chi\bmod{q_{\chi}}\colon \textup{$\chi$ primitive and real, $1 \leq q_{\chi} \leq Q$}\}.
\]
\begin{thm}[Page] \label{thm: Page}
There exists an absolute constant $c>0$ such that the following holds. For any $Q \geq 2$, we have
\begin{align*}
    \#\{\textup{$\chi\in \mathcal{S}(Q)$: $L(s,\chi)$ has a real zero in $[1-c(\log Q)^{-1},1)$}\}\leq 1.
\end{align*}
\end{thm}
Theorem \ref{thm: Page} allows one to consider only a single exceptional Landau--Siegel zero. The effect of this single exceptional zero can then be mitigated, at least in some kinds of number-theoretic arguments. For example, in \cite{May2016}, sieve weights are chosen to be supported on integers not divisible by the largest prime factor of the conductor of the exceptional character. This allows for suitable control on primes in arithmetic progressions beyond the typical Siegel-Walfisz range (see \cite[Chapter 22]{Davenport}).

Given the usefulness of Theorem \ref{thm: Page}, one naturally wonders whether the result can be strengthened and improved. Our goal here is to show that Theorem \ref{thm: Page} may be refined in a way analogous to how \eqref{Eq: ZFR BTZ} was established, if the zeros of Dirichlet $L$-functions behave suitably. 

\begin{hypothesis*}[$H_f$]
     Let $f(x)$ be a positive real-valued increasing function with $\lim_{x \to \infty} f(x) = \infty$. For $\chi \bmod q_{\chi}$ primitive and real, we say $L(s,\chi)$ satisfies Hypothesis $H_f$ provided all the zeros of $L(s,\chi)$ inside the disk $| s-1 | < 1/f(q_{\chi})$, if any, are real.
\end{hypothesis*}


Our main result is as follows.
\begin{thm} \label{thm: Main Theorem}
Let $\nu : \mathbb{R} \to \mathbb{R}_{>0}$ be a positive, real-valued function satisfying the following conditions:
\begin{itemize}
\item[1. ]  $\lim_{x \to \infty} \nu(x) = 0$.
\item[2. ] For $x > 1$, the function $f(x) := (\log x)^{\nu(x)}$ is increasing, and $\lim_{x \to \infty} f(x) = \infty.$
\end{itemize}
Let $\varepsilon>0$. There exists an effectively computable constant $Q_0 = Q_0(\nu,\varepsilon) \geq 1$ depending at most on $\nu$ and $\varepsilon$ with the following property. For any $Q \geq Q_0$ for which Hypothesis $H_{f}$ holds for all $\chi \in \mathcal{S}(Q^2)$, we have
\begin{align}\label{eq 2}
\#\{\textup{$\chi\in \mathcal{S}(Q)$: $L(s,\chi)$ has a real zero in $[1-(\log Q)^{-\varepsilon},1)$}\}\leq 1.
\end{align}
\end{thm}

For an explicit permissible expression for $Q_0(\nu,\varepsilon)$, see \eqref{Eq: Explicit dependence}. We make a few remarks.
\begin{rem}
We observe that Hypothesis $H_f$ allows the existence of Landau--Siegel zeros. At the same time, it allows non-real zeros of $L(s,\chi)$ to be arbitrarily close to $s=1$ as $q_{\chi}$ tends to infinity. Note that Hypothesis $H_f$ in Theorem \ref{thm: Main Theorem} is weaker than Hypothesis $H_{\delta}$, since for $q_{\chi}$ sufficiently large and any fixed $\delta>0$, we have $1/f(q_{\chi}) <\delta$. In other words, under Hypothesis $H_f$, the non-real zeros of $L(s,\chi)$ lie outside a shrinking neighborhood of $s=1$, unlike in Hypothesis $H_{\delta}$, which forces non-real zeros to lie outside a fixed $\delta$-neighborhood of $s=1$.
\end{rem}
\begin{rem}
It follows from Heath--Brown's zero density estimate in \cite[Theorem 3]{Heath-Brown} that if the $\chi \in \mathcal{S}$ are ordered by conductor $q_\chi$, then a density one subset of $\chi \in \mathcal{S}$ satisfy $H_\delta$, as well as Hypothesis $H_f$ from Theorem \ref{thm: Main Theorem}. In contrast, the hypothesis used in the work of Sarnak and Zaharescu has not been verified for any $\chi \in \mathcal{S}$ yet. 
\end{rem}

\begin{rem}\label{Rem: Bounds on Exceptional Characters}
It is shown in \cite[p. 94]{Davenport} that if $q_1, q_2, \ldots$ forms a sequence of positive integers $q$ with the property that there is a real primitive $\chi \bmod q$ for which $L(s, \chi)$ has a real zero $\beta$ satisfying \eqref{Eq: Classical ZFR}, then one must have
$$
q_{j+1}>q_j^2,
$$
which implies that the number of exceptional primitive quadratic characters with moduli up to $Q$ is $O(\log \log Q)$. Under Hypothesis $H_f$ for all $\chi \in \mathcal{S}(Q^2)$, this bound can be improved to $O_{f}( \log \log \log Q)$. For details, see Section \ref{sec:compare contrast}.
\end{rem}

Finally, we mention that versions of Theorem \ref{thm: Main Theorem} can be obtained for more general families of $L$-functions, such as Dirichlet $L$-functions attached to characters of higher order, and $L$-functions associated to Hecke--Maass cusp forms (see, e.g., Theorems 4.1 and 4.2 in \cite{basak2024remarks} for extensions in a similar context). Later, in Section \ref{sec:compare contrast}, we give further discussion regarding the similarities and differences between Theorem \ref{thm: Main Theorem} and the results from \cite{basak2024remarks}.

\section{Proof of Theorem \ref{thm: Main Theorem}}\label{sec : Proof of Main Theorem}

We apply Tur\'an's power sum method (see \cite[Chap. 9, Lem. 2]{M1994}) to prove Theorem \ref{thm: Main Theorem}. We state the lemma here.

\begin{lem}\label{Turan's Power Sum}
Let $(u_n)_{n \geq 1}$ be a sequence of complex numbers with $\sup_{n} \lvert u_n \rvert =\lvert u_1 \rvert$. Set
\begin{align}\label{U Definition}
U=\frac{\sum_{j \geq 1} \left|u_j\right|}{\left|u_1\right|}.
\end{align}
Then there exists an integer $r$ with $1 \leq r \leq 24 U$ such that
$$
\operatorname{Re} \sum_{j \geq 1} u_j^r \geq \frac{1}{8} \left|u_1\right|^r .
$$
\end{lem}

\begin{proof}[Proof of Theorem \ref{thm: Main Theorem}] Let $\varepsilon>0$ be given. Let $Q_0=Q_0(\nu, \varepsilon) \geq 1$ be sufficiently large, depending at most on $\nu$ and $\varepsilon$, and which is explicitly computable in terms of $\nu$ and $\varepsilon$. 

We proceed by way of contradiction. Therefore, suppose there exists some $Q \geq Q_0$ and two distinct real primitive Dirichlet characters $\chi_1$ and $\chi_2$ of conductors 
\begin{align*}
q_1 \leq Q \quad \textrm{and} \quad q_2 \leq Q,
\end{align*} 
such that the following holds: $L(s,\chi_1)$ and $L(s,\chi_2)$ have real zeros $\beta_1$ and $\beta_2$ respectively satisfying
\begin{align}\label{chi EZ}
\beta_1 > 1-(\log Q)^{-\varepsilon}\quad \textrm{and} \quad
\beta_2 > 1-(\log Q)^{-\varepsilon}.
\end{align}
Assume without loss of generality that $q_1 \geq q_2$. We break our argument into several steps.

\subsection{Log-Derivative and Non-Negativity.} We begin with the function
\begin{align*}
F(s)=\zeta(s) L\left(s, \chi_1\right) L(s, \chi_2) L\left(s, \psi\right),
\end{align*}
where $\psi$ is the primitive quadratic character that induces $\chi_1\chi_2$. For $\operatorname{Re}(s)>1$, one has
\begin{align*}
    -\frac{F^{\prime}}{F}(s)=\sum_{n \geq 1} \frac{\Lambda(n)(1+\chi_1(n)+\chi_2(n)+\psi(n))}{n^s} .
\end{align*}
On the other hand, we have the partial fraction expansion
\begin{align}\label{Hadamard Product}
-\frac{F^{\prime}}{F}(s)=\frac{1}{s-1}-\sum_{F(\rho)=0}\left(\frac{1}{s-\rho}+\frac{1}{\rho}\right)+B,
\end{align}
for some constant $B$. Our strategy is to apply Tur\'an's power sum method. However, we do not apply it directly. First, we will raise our quantities of interest to a convenient high power $\ell$, and then apply Tur\'an's power sum method. This leads us to consider multiples of $\ell$, say $k\ell$. With this in mind, we start with \eqref{Hadamard Product} and differentiate both sides $k\ell-1$ times to obtain
\begin{align}\label{Differentiation}
\frac{1}{(s-1)^{k\ell}}-\sum_{F(\rho)=0} & \frac{1}{(s-\rho)^{k\ell}} =\frac{1}{(k\ell-1) !} \sum_{n \geq 1} \frac{\Lambda(n)(1+\chi_1(n)+\chi_2(n)+\psi(n)) (\log n)^{k\ell-1}}{n^s} .
\end{align}
We choose $s=1+\eta$ in \eqref{Differentiation}, for some $\eta>0$ that we will optimize later. Since the right-hand side is non-negative, by taking the real part on both sides of \eqref{Differentiation}, one has
\begin{align}\label{Non-Negativity}
\frac{1}{\eta^{k\ell}}-\operatorname{Re} \sum_{\substack{F(\rho)=0 }} \frac{1}{(1+\eta-\rho)^{k\ell}} \geq 0 .
\end{align}
Next, we rearrange \eqref{Non-Negativity} as
\begin{align}\label{Rearrangement}
\frac{1}{\eta^{k\ell}}-\frac{1}{(1+\eta-\beta_1)^{k\ell}}-\sum_{\substack{F(\rho)=0 \\ \rho \textrm{ real, } \rho \neq \beta_1, \beta_2}} \frac{1}{(1+\eta-\rho)^{k\ell}} \geq \frac{1}{(1+\eta-\beta_2)^{k\ell}} + \operatorname{Re} \sum_{\substack{F(\rho)=0 \\ \operatorname{Im}{\rho} \neq 0}} \frac{1}{(1+\eta-\rho)^{k\ell}} .
\end{align}
Since the terms, if any, in the sum on the left hand side of \eqref{Rearrangement} are positive, we remove them to obtain
\begin{align}\label{Final Simplification}
\frac{1}{\eta^{k\ell}}-\frac{1}{(1+\eta-\beta_1)^{k\ell}} \geq \frac{1}{(1+\eta-\beta_2)^{k\ell}} + \operatorname{Re} \sum_{\substack{F(\rho)=0 \\ \operatorname{Im}{\rho} \neq 0}} \frac{1}{(1+\eta-\rho)^{k\ell}} .
\end{align}
\subsection{Application of Tur\'an's Power Sum Method.} In what follows, $B_0$, $B_1$, $B_2, \dots$ are absolute effectively computable constants. We choose
\begin{align}\label{Eta and l choice}
\eta= \frac{1}{e \cdot f(q_1q_2)}+\beta_2-1, \quad \textrm{and} \quad \ell = \lceil \log \log Q \rceil.  
\end{align} 
Let
\begin{align*}
u_1 = (1+\eta-\beta_2)^{-\ell}.
\end{align*} 
Also, let us write the numbers $(1+\eta-\rho)^{-\ell}$ as a sequence of complex numbers $ u_2, u_3, \dots$, arranged such that $\left|u_2\right| \geq\left|u_3\right| \geq \cdots$. We will require $\left|u_1\right| \geq \left|u_2\right|$ to apply Tur\'an's power sum method. To see this, note that for $Q \geq Q_0$, we have
\begin{align}\label{beta-delta}
\max \{ 1-\beta_1, 1-\beta_2 \}\leq (\log Q)^{-\varepsilon} \leq \frac{1}{2(\log (Q^2))^{\nu(Q^2)}} \leq \frac{1}{2f(q_1q_2)}.
\end{align}
Also, it follows from our choice of $\eta$ in \eqref{Eta and l choice} that
\begin{align}\label{eta-delta}
0 < \eta \leq \frac{3}{4 f(q_1q_2)}.
\end{align}
The upper bound is immediate. The lower bound follows from the observation that for $Q \geq Q_0$,
\begin{align}\label{verifying eta bound}
\eta \geq \frac{1}{e} \cdot \frac{1}{(\log (Q^2))^{\nu(Q^2)}} - (\log Q)^{-\varepsilon} \geq \frac{1}{4} \cdot \frac{1}{(\log (Q^2))^{\nu(Q^2)}}>0.
\end{align}
Hence for any $\rho$ with $\operatorname{Im}(\rho) \neq 0$ such that $F(\rho)=0$, we apply Hypothesis $H_f$ in conjunction with \eqref{beta-delta} and \eqref{eta-delta} to deduce that
\begin{align*}
\lvert 1+\eta-\rho \rvert \geq \sqrt{\eta^2+\frac{1}{f(q_1q_2)^2}} \geq \eta+\frac{1}{2f(q_1q_2)} \geq  1+\eta-\beta_2.
\end{align*}
This ensures that $\left|u_1\right| \geq \left|u_2\right|$ and we are ready to apply Lemma \ref{Turan's Power Sum}. In our case, we have
\begin{align}\label{U defn}
U = \frac{\left|u_1\right|+\sum_{t \in \mathbb{Z}} U_t}{\left|u_1\right|} , 
\end{align}
where 
\[
U_t = \sum_{\substack{j \geq 2 \\ t < \operatorname{Im} \rho_j \leq t+1 }} \left |u_j\right| = \sum_{\substack{j \geq 2 \\ t < \operatorname{Im} \rho_j \leq t+1 }} \lvert 1+\eta-\rho_j \rvert^{-\ell} .
\]
First consider the case when $t  \geq 2$. Note that for all $u_j$ with $t < \operatorname{Im} \rho_j \leq t+1$, we have $\left |u_j\right| \leq t^{-\ell}$. Therefore, using standard zero counting arguments from \cite[Prop. 5.7]{IK2004} (see also Davenport \cite[p. 99]{Davenport}), the number of zeros $\rho_j$ with $t < \operatorname{Im} \rho_j \leq t+1$ is $\leq B_1\left (\log Q +\log(\lvert t \rvert +2)\right )$. Therefore, for $t  \geq 2$, we obtain
\begin{align}\label{t >2}
U_t \leq B_1 \left (\frac{\log Q +\log(\lvert t \rvert +2)}{t^{\ell}} \right),
\end{align}
A concomitant argument shows that for  $t\leq -3$,
\begin{align}\label{t <-3}
U_t \leq B_2 \left (\frac{\log Q +\log(\lvert t \rvert +2)}{\lvert t+1 \rvert^{\ell}} \right).
\end{align}
Finally, when $-2\leq t \leq 1,$ we have $\left |u_j\right| \leq f(q_1q_2)^{\ell}$ for all $j \geq 2$. Hence in this case, we obtain
\begin{align}\label{t small}
U_t \leq B_3 \cdot (\log Q) \cdot f(q_1q_2)^{\ell}.
\end{align}
Therefore, putting together \eqref{t >2}, \eqref{t <-3} and \eqref{t small}, we arrive at
\begin{align}\label{Bounding U_t}
    \sum_{t \in \mathbb{Z}}U_t &\leq B_4\bigg  ( (\log Q) \cdot f(q_1q_2)^{\ell} +\sum_{t \geq 2}\frac{\log Q +\log(t+3)}{t^{\ell}}\bigg) \notag \\
    & \leq B_5 \cdot (\log Q) \cdot f(q_1q_2)^{\ell}.
\end{align}
Finally combining \eqref{Eta and l choice}, \eqref{U defn} and \eqref{Bounding U_t}, we arrive at
\begin{align}\label{Bounding U}
U \leq 1+B_5 \cdot (\log Q) \cdot \left ( 1+\eta-\beta_2 \right)^{\ell} \cdot (f(q_1q_2))^{\ell} \leq B_6.
\end{align}
Applying Lemma \ref{Turan's Power Sum}, it follows that, for some $1 \leq r \leq 24 U \leq B_7$, we have
\begin{align*}
\operatorname{Re} \sum_{j \geq 1} u_j^r \geq \frac{1}{8} \left|u_1\right|^r.
\end{align*}
Substituting this estimate in \eqref{Final Simplification} and choosing $k=r$, we have
\begin{align}\label{eq prefinal}
\frac{1}{\eta^{r\ell}}-\frac{1}{(1+\eta-\beta_1)^{r\ell}} &\geq \frac{1}{(1+\eta-\beta_2)^{r\ell}} + \operatorname{Re} \sum_{\substack{F(\rho)=0 \\ \operatorname{Im}{\rho} \neq 0}} \frac{1}{(1+\eta-\rho)^{r\ell}} \notag \\
&\geq \frac{1}{8(1+\eta-\beta_2)^{r\ell}} .
\end{align}
\subsection{Concluding Arguments.} Multiplying \eqref{eq prefinal} through by $\eta^{r\ell}$, we have
\begin{align}\label{Rewriting}
    1-\left(1- \frac{1-\beta_1}{1+\eta-\beta_1} \right)^{r \lceil \log \log Q \rceil} \geq \frac{1}{8}\left( 1-\frac{1-\beta_2}{1+\eta-\beta_2} \right)^{r\lceil \log \log Q \rceil}.
\end{align}
We now make use of Bernoulli's inequality (see \cite[Lemma 3.4]{basak2024remarks}): if $0<a<1$ and $b>1$, then $ab>1-(1-a)^b$. We choose 
\[
a = \frac{1-\beta_1}{1+\eta-\beta_1} \in (0,1) \quad \textrm{and} \quad b = r \lceil \log \log Q \rceil.
\]
Using Bernoulli's inequality to the left hand side of \eqref{Rewriting}, we have
\begin{align*}
r \lceil \log \log Q \rceil \cdot  \frac{1-\beta_1}{1+\eta-\beta_1} & \geq \frac{1}{8}\left( 1-\frac{1-\beta_2}{1+\eta-\beta_2} \right)^{r\lceil \log \log Q \rceil}
\end{align*}
Therefore, we obtain
\begin{align}\label{Rewriting 2}
1-\beta_1 &\geq  \frac{1}{8} \cdot \frac{1+\eta-\beta_1}{r \lceil \log \log Q \rceil}\left( 1-\frac{1-\beta_2}{1+\eta-\beta_2} \right)^{r \lceil \log \log Q \rceil } \notag \\
&\geq \frac{\eta}{16 r \log \log Q}\left( 1-\frac{1-\beta_2}{1+\eta-\beta_2} \right)^{2r \log \log Q }.
\end{align}
Using the inequality $1-\beta_1 \leq (\log Q)^{-\varepsilon}$ in \eqref{Rewriting 2}, we arrive at
\begin{align}
\frac{1}{(\log Q)^{\varepsilon}} &\geq \frac{\eta}{16 r \log \log Q}  \cdot ( \log Q)^{2 r\log\left(1- \frac{1-\beta_1}{1+\eta-\beta_1} \right)} \label{theta step}.
\end{align}
By \eqref{verifying eta bound}, we have for  $Q\geq Q_0$,
\begin{align}\label{removing loglog}
 \frac{16 r \log \log Q}{\eta} \leq B_8 (\log \log Q) \cdot (\log Q^2)^{\nu(Q^2)} \leq (\log Q)^{\frac{\varepsilon}{2}}.
\end{align}
Therefore substituting \eqref{removing loglog} in \eqref{theta step}, we derive 
\begin{align}\label{theta step 2}
\frac{1}{(\log Q)^{\varepsilon}} \geq ( \log Q)^{2r\log\left(1- \frac{1-\beta_1}{1+\eta-\beta_1} \right)-\frac{\varepsilon}{2}}.
\end{align}
Finally, taking logarithms on both sides of \eqref{theta step 2}, we obtain
\begin{align}\label{Epsilon Bound}
\varepsilon \leq 
-4r\log\left(1- \frac{1-\beta_2}{1+\eta-\beta_2} \right). 
\end{align}
By \eqref{verifying eta bound}, we have
\[\frac{1-\beta_2}{1+\eta-\beta_2} \leq \frac{1}{2}.
\] 
Next, we use the inequality 
\[
-\log (1-y) \leq 2y, \quad 0 \leq y \leq \frac{1}{2},\]
in \eqref{Epsilon Bound} with the choice
\[ y = \frac{1-\beta_2}{1+\eta-\beta_2}.
\]
This yields
\begin{align}
\varepsilon  &\leq \frac{8r(1-\beta_2)}{1+\eta-\beta_2}  \leq \frac{8r}{\eta} \cdot \frac{1}{(\log Q)^{\varepsilon}} \leq  \frac{B_8 (\log Q^2)^{\nu(Q^2)}}{(\log Q)^{\varepsilon}}, \label{final}
\end{align}
where in the right-most inequality, we again use \eqref{verifying eta bound}. Now since $Q \geq Q_0$, we now choose $Q_0$ sufficiently large depending at most on $\nu$ and $\varepsilon$. For instance, if $\lvert \nu(Q^2) \rvert \leq \varepsilon/2$ for all $Q \geq Q_{\nu, \varepsilon}$ say, we let
\begin{align}\label{Eq: Explicit dependence}
Q \geq Q_0(\nu,\varepsilon)  = \max \bigg \{  Q_{\nu,\varepsilon},\exp \left ( \frac{ B_8 \, 2^{\frac{\varepsilon}{2}}}{\varepsilon}\right)^{\frac{2}{\varepsilon}}\bigg \}
\end{align}
to arrive at a contradiction to \eqref{final}, and thereby, to our initial assumption. This concludes the proof of the theorem.
\end{proof}

\section{Discussion and contrast between results}\label{sec:compare contrast}

We now comment on the subtle distinction between the proof of Theorem \ref{thm: Main Theorem} and that of \cite[Theorem 1.1]{basak2024remarks}. For the benefit of the reader, we briefly recall their result here. Let Hypothesis $H_{\delta}$ be as in Section \ref{sec : Introduction}.

\begin{thm}\label{thm: BTZ}
Fix $0<\delta<1 / 10$. If $H_\delta$ is true, then for all $\varepsilon>0$, there exists an effectively computable constant $q_0=q_0(\delta, \varepsilon)>0$ such that
\[
\#\{\textup{$\chi\in \mathcal{S}$: $q_{\chi}\geq q_0$ and $L(s,\chi)$ has a real zero in $[1-(\log q_{\chi})^{-\varepsilon},1)$}\}\leq 1.
\]
\end{thm}

In the proof of Theorem \ref{thm: BTZ}, it is essential to establish an analogue of \eqref{beta-delta} in order to apply Tur\'an's power sum method. For instance, under Hypothesis $H_{\delta}$ and assuming $q_1 \geq q_2$, we require that
\begin{align}\label{Eq : Crucial 1}
1-\beta_2 \leq (\log q_2)^{-\varepsilon} \leq \delta.
\end{align}
This condition is satisfied for sufficiently large $q_2$ depending on $\delta$. 

In contrast, under Hypothesis $H_{f}$, we require the stronger condition
\begin{align}\label{Eq : Crucial}
1-\beta_2 \leq (\log q_2)^{-\varepsilon} \leq \frac{1}{2 f(q_1q_2)}.
\end{align}
Note, however, that \eqref{Eq : Crucial} may fail when $q_1$ is significantly larger than $q_2$. For example, if $f(x) = \log \log x$ and we set $q_1 = \lceil q_2^{-1} \exp(q_2) \rceil$, then \eqref{Eq : Crucial} is no longer satisfied. To elaborate further, applying Tur\'an's power sum method requires some control over the real zero $\beta_2$ from $L(s,\chi_2)$ and all the non-real zeros arising from $F(s)$, with $F(s)$ as in Section \ref{sec : Proof of Main Theorem}. Under Hypothesis $H_\delta$, the non-real zeros from $F(s)$, in particular, those from $L(s,\psi)$, are still at a distance $\geq\delta$ from $s=1$. Thus, by taking $q_2$ sufficiently large in terms of $\delta$ and choosing $\eta$ appropriately, we ensure that for any complex zero $\rho$,
\[
\lvert 1+\eta -\rho \rvert \geq 1+\eta - \beta_2.
\]
On the other hand, with Hypothesis $H_f$ and $f(x) = \log \log x$, say, the non-real zeros from $L(s,\psi)$ could be at a distance as small as $1/f(q_1q_2) \approx 1/(\log q_2)$ from $s=1$. This is too close to $s=1$, or even $s=1+\eta$, compared to $\beta_2$ which could be at a distance $1/(\log q_2)^{\varepsilon}$ from $s=1$. Therefore, Tur\'an's power sum method fails in this situation. To overcome this issue, it is necessary to impose an upper bound on $q_1$, as in the framework of Theorem \ref{thm: Main Theorem}.

Another subtle distinction between Theorem \ref{thm: Main Theorem} and Theorem \ref{thm: BTZ} is that the conclusion of Theorem \ref{thm: Main Theorem} involves $\log Q$, which can make direct comparison to the zero-free region \eqref{Eq: Classical ZFR} difficult if $q$ is small compared to $Q$. However, if one works with characters having conductors in a dyadic interval $Q < q \leq 2Q$, say, then the conclusion of Theorem \ref{thm: Main Theorem} can be directly compared to the zero-free region, as in Theorem \ref{thm: BTZ}.

We now discuss Remark \ref{Rem: Bounds on Exceptional Characters}. Suppose $q_1, q_2, \dots$ form a sequence of exceptional moduli. We wish to bound the number of such $q$'s up to $Q$. Let $A>0$ be a fixed large constant and $\varepsilon$ be small enough depending only on $A$. Assume Hypothesis $H_f$ for all $\chi \in \mathcal{S}(Q^2)$. We may assume without loss of generality that $q_1>Q_0(\nu,\varepsilon)$. We apply Theorem \ref{thm: Main Theorem} with the choice
\[
Q_1 = \exp( (\log q_1)^A).
\]
It follows that
\begin{align}\label{Eq: Set Counting}
\#\{\textup{$\chi\in \mathcal{S}(Q_1)$: $L(s,\chi)$ has a real zero in $[1-(\log Q_1)^{-\varepsilon},1)$}\}\leq 1.
\end{align}
Since the exceptional character modulo $q_1$ is already counted in the left hand side of \eqref{Eq: Set Counting}, it follows that
\[
q_2 > Q_1 = \exp( (\log q_1)^A).
\]
Repeatedly arguing as above, one obtains
\[
q_{j+1} > \exp( (\log q_1)^A), \quad j \in \mathbb{N},
\]
which shows that the number of exceptional characters with moduli up to $Q$ is $O_{\nu}(\log \log \log Q)$.
\section*{Acknowledgments}
We thank the anonymous referee for helpful suggestions and feedback. We also thank Jesse Thorner and Alexandru Zaharescu for helpful comments. The second author is partially supported by the National Science Foundation (DMS-2418328) and the Simons Foundation (MPS-TSM-00007959).

\bibliographystyle{plain}
\bibliography{references}

\section*{Data availability statement}

No data are associated with this article.

\section*{Conflict of interest statement}

The authors have no financial or proprietary interests in any material discussed in this article.

\end{document}